\documentclass[a4, 12pt]{amsart}

\usepackage[all]{xy}
\usepackage{tikz-cd} 

\usepackage{color}
\usepackage{xcolor} 
\usepackage{amssymb, bm, bbm}
\usepackage{hyperref}
\hypersetup{
	colorlinks=true,
	linkcolor=blue,
	citecolor=magenta}

\newtheorem{thm}{Theorem}[section]
\newtheorem{prop}[thm]{Proposition}
\newtheorem{lem}[thm]{Lemma}
\newtheorem{cor}[thm]{Corollary}

\numberwithin{equation}{section}

\newtheorem{bigthm}{Theorem}

\newtheorem{bigconj}{Conjecture}

\theoremstyle{definition}

\newtheorem{ex}[thm]{Example}

\newtheorem{rem}[thm]{Remark}

\def\cI{\mathcal{I}}
\def\Z{\mathbb{Z}}
\def\Q{\mathbb{Q}}  \def\C{\mathbb{C}}
 \def\R{\mathbb{R}}

  \def\P{\mathbb{P}}

  \def\cO{\mathcal{O}}
\def\cF{\mathcal{F}}
\def\P{\mathbb{P}}
\def\cI{\mathcal{I}}

\def \log{\mathrm{log}}
\def \det{\mathrm{det}}

\def \dim{\mathrm{dim}}
\def \rank{\mathrm{rank}}

\def \hom{\mathrm{Hom}}

\def\cS{\mathcal{S}}
\def\cQ{\mathcal{Q}}

\makeatletter
\newcommand*{\da@rightarrow}{\mathchar"0\hexnumber@\symAMSa 4B }
\newcommand*{\da@leftarrow}{\mathchar"0\hexnumber@\symAMSa 4C }
\newcommand*{\xdashrightarrow}[2][]{%
  \mathrel{%
    \mathpalette{\da@xarrow{#1}{#2}{}\da@rightarrow{\,}{}}{}%
  }%
}
\newcommand{\xdashleftarrow}[2][]{%
  \mathrel{%
    \mathpalette{\da@xarrow{#1}{#2}\da@leftarrow{}{}{\,}}{}%
  }%
}
\newcommand*{\da@xarrow}[7]{%
  \sbox0{$\ifx#7\scriptstyle\scriptscriptstyle\else\scriptstyle\fi#5#1#6\m@th$}%
  \sbox2{$\ifx#7\scriptstyle\scriptscriptstyle\else\scriptstyle\fi#5#2#6\m@th$}%
  \sbox4{$#7\dabar@\m@th$}%
  \dimen@=\wd0 %
  \ifdim\wd2 >\dimen@
    \dimen@=\wd2 %
  \fi
  \count@=2 %
  \def\da@bars{\dabar@\dabar@}%
  \@whiledim\count@\wd4<\dimen@\do{%
    \advance\count@\@ne
    \expandafter\def\expandafter\da@bars\expandafter{%
      \da@bars
      \dabar@ 
    }%
  }%
  \mathrel{#3}%
  \mathrel{%
    \mathop{\da@bars}\limits
    \ifx\\#1\\%
    \else
      _{\copy0}%
    \fi
    \ifx\\#2\\%
    \else
      ^{\copy2}%
    \fi
  }%
  \mathrel{#4}%
}
\makeatother

\usepackage{geometry}
\begin{document}

\title[Albanese map for K\"ahler manifolds with nef anticanonical bundle]
{Albanese map for K\"ahler manifolds \\ with nef anticanonical bundle}

\author{Philipp NAUMANN}
\address{Institut fur Mathematik, Universit\"at Bayreuth, 95440 Bayreuth, Germany}
\email{{\tt philipp.naumann@uni-bayreuth.de}}

\author{Xiaojun WU}
\address{Institut fur Mathematik, Universit\"at Bayreuth, 95440 Bayreuth, Germany.}
\email{{\tt xiaojun.wu@uni-bayreuth.de}, {\tt xiaojun.wu@univ-cotedazur.fr}}

\date{\today, version 0.01}

\renewcommand{\subjclassname}{%
\textup{2010} Mathematics Subject Classification}


\maketitle

\begin{abstract}
We study the structure of the Albanese map for K\"ahler manifolds with nef anticanonical bundle. First, we give a result for fourfolds whose Albanse torus is an elliptic curve. In the general case of any dimension, we look at two cases: The general fiber of the Albanese map is a Calabi-Yau manifold or a projective space. In the first case, we show that the manifold itself must be Calabi-Yau. In the second case, we give a more topological proof of a result by Cao and H\"oring which says that the manifold must be a projectivization of a numerically flat vector bundle.
\end{abstract}


\section{Introduction}
In this work, we study the Albanese map of compact K\"ahler manifolds with nef anticanonical line bundle.
The conjectural picture is the following:

\begin{bigconj}
Let $X$ be a compact K\"ahler manifold with nef anticanonical line bundle.
Then the Albanese morphism of $X$ is a locally trivial fibration with connected fibers.
\end{bigconj}

The conjecture is fully proven if $X$ is projective in \cite{Cao19}.
However in the compact K\"ahler case, it is still widely open and it was only proven recently in dimension three \cite{MW23}. 
The main difficulty in the K\"ahler case arises from the fact that the Albanese morphism is no longer projective, since the fibers are not necessarily projective. Therefore, the theory of positivity of direct images - a key technique in the projective case - does not apply directly. Although it is conjectured by Campana that the MRC quotient stated in Conjecture 2 below is a projective morphism with the base being Calabi-Yau, the Albanese map can only be projective if the Hyperk\"ahler components in the Beauville-Bogomolov decomposition of the MRC quotient of $X$ are all projective.

After revealing the structure of the Albanese map for K\"ahler threefolds, a next step would be to look at dimension 4. 
Using the recent developements of 4-dimensional K\"ahler MMP in \cite{DHP23}, we first give a result about the Albanese morphism for fourfolds:

\begin{bigthm}
Let $X$ be a non-projective compact K\"ahler fourfold with $-K_X$ nef such that $\tilde{q}(X)=q(X)=1$. 
Then, up to a finite \'etale cover, $X$ is either the product of a K3 surface and a projectivization of a rank two numerically flat vector bundle over an elliptic curve or the product of a simply connected Calabi-Yau threefold with an elliptic curve.
\end{bigthm}

In the general case of arbitrary dimension, we mainly deal with two cases: In the first case, we assume that the fibers of the Albanese map are simply connected Calabi-Yau manifolds.  In this setting we have the following result:

\begin{bigthm}
Let $X$ be a non-uniruled compact K\"ahler manifold with $-K_X$ nef and $\tilde{q}(X)=q(X)$. If the general fiber of the Albanese map is a Calabi-Yau manifold we have $c_1(X)=0$.
\end{bigthm}
Here, the key idea is to use the close relation between the curvature of the relative canonical bundle and deformation theory in order to replace \cite{BDPP} which is not known in the K\"ahler case.
The second result covers the case where the fibers of the Albanese map are projective spaces:
\begin{bigthm}
Let $X$ be a compact K\"ahler manifold such that 
$-K_X$ is nef.
Assume that the general fiber of the Albanese map is $\P^{r-1}$.
Then $X \simeq \P(E)$ for some numerically flat vector bundle $E$ of rank $r$ over $A(X)$ up to some finite \'etale cover.
\end{bigthm}
This result was  known previously by \cite{CH17}. Here we give a completely different and more topological proof of it.
The key observation is that using the results of \cite{Cao13} and \cite{ARM14}, the Albanese morphism is smooth in codimension 2 and the manifold is the projectivization of some vector bundle in codimension 1.
Using the result of \cite{BS94}, we can extend the vector bundle to a numerically flat vector bundle over the Albanese torus by a careful control of its first and second Chern classes using the codimension condition.

When the fibers are rational surfaces, we prove a structure theorem assuming the smoothness of the Albanese morphism and a condition on the intersection number. It is a special case of \cite{CH17}.
However, we find that its proof has its own interest using only deformation theory, cf. also \cite[Corollary 3.4]{PS98}.
\begin{bigthm}
 Let $X$ be a uniruled compact K\"ahler manifold of dimension $n$ such that
$\tilde{q}(X)=q(X)=n-2$ and
$-K_X$ is nef.
Assume that the Albanese morphism $\alpha: X \to A(X)$ is smooth.
Moreover, assume that for any fiber $F$ of $\alpha$,
$$(K_F)^2 \geq 2.$$
Then, up to some finite \'etale cover, the possibilities of $X$ are one of the following:
\newpage
\begin{enumerate}
 \item $X \simeq \P(E)$ for some numerically flat vector bundle (of rank 3) over $A(X)$;
 \item $X \simeq \P(E_1) \times_{A(X)} \P(E_2)$ for some numerically flat vector bundles $E_1$ and $E_2$ (of rank 2) over $A(X)$;
 \item we have a factorisation
$$X \xrightarrow{\gamma_1} X_1 \cdots \xrightarrow{\gamma_k} X_k \xrightarrow{\rho}  A(X)$$
with $\gamma_j$ blow-ups of  \'etale multi-sections, $X_k \simeq \P(E_1) \times_{A(X)} \P(E_2)$ for some vector bundles $E_1,E_2$ over $A(X)$ or $X_k \simeq \P(E)$ for some vector bundle over $A(X)$.
\end{enumerate}
Moreover, the Albanese morphism is locally constant in this case.
\end{bigthm}
The concept of locally trivial morphism is introduced in \cite{MW}.

At the end, we discuss the algebraic approximation of compact K\"ahler manifolds with nef anticanonical line bundle.
Recall that $X$ is called to be algebraically approximated if there exists a deformation of $\mathcal{X} \to \Delta$ over the unit disk such
that the central fiber $X_0$ is $X$, and there exists a sequence $t_i \to 0$ in $\Delta$ such that all
the fibers $X_{t_i}$ are projective.

We have the following structural conjecture which is proven up to dimension 3 in \cite{MW23}. It is settled completely in the projective case in \cite{CH19}.
\begin{bigconj}\label{conj-main}
Let $X$ be a compact K\"ahler manifold with the nef anti-canonical bundle. 
There exists a fibration $\varphi: X \to Y$ with the following properties: 
\begin{itemize}{}
\item $\varphi: X \to Y$ is a locally constant fibration (i.e. $\varphi:X \simeq (\tilde{Y} \times F)/ \pi_1(Y) \to Y=\tilde{Y}/\pi_1(Y) $ is induced by the natural projection where $\tilde{Y}$ is the universal cover of $Y$, $F$ the fiber of $\varphi$ together with a representation $\rho: \pi_1(X) \to \operatorname{Aut}(F)$); 
\item $Y$ is a compact K\"ahler manifold with $c_{1}(Y)=0$; 
\item $F$ (being the fiber of $\varphi: X \to Y$)  is rationally connected. 
\end{itemize}
\end{bigconj}

Assuming this conjecture, we can show the following result:

\begin{bigthm}
If Conjecture 2 holds true, then
any compact K\"ahler manifold $X$ with nef anticanonical divisor can be algebraically approximated.
\end{bigthm}

\subsection*{Acknowledgments}\label{subsec-ack}
The authors would like to thank Fr\'ed\'eric Campana, Patrick Graf, Shin-ichi Matsumura, Mihai P\u{a}un and Thomas Peternell for kindly answering our questions.
The second author was supported by the DFG Projekt
\emph{Singul\"are hermitesche Metriken f\"ur Vektorb\"undel und Erweiterung kanonischer Abschnitte} managed by Mihai P\u{a}un. \\

\newpage

\section{Structure of the Albanese map for 4-folds when the base is a curve}
\begin{prop} \label{Prop1}
Let $X$ be a non-projective compact K\"ahler fourfold with $-K_X$ nef such that $\tilde{q}(X)=q(X)=1$.
Then up to a finite \'etale cover, $X$ is either the product of a K3 surface and a projectivization of a rank two numerically flat vector bundle over an elliptic curve or the product of a simply connected Calabi-Yau threefold with an elliptic curve.
\end{prop}
\begin{proof}
Let $\alpha: X \to A(X)$ be the Albanese morphism which is smooth by \cite{Cao13}, since the base is of dimension one.
In particular, by \cite[ Proposition 3.12]{DPS94}, $\tilde{q}(F)=0$ where $F$ is any fiber. 
Note that if $K_X$ is $\Q-$effective, $c_1(X)=0$ and $X$ is the product of a simply connected Calabi-Yau threefold with an elliptic curve up to a finite \'etale cover.
In the following, we always assume that $K_X$ is not $\Q-$effective.
Up to a finite \'etale cover, we can assume that $\pi_1(X)$ is free abelian, see \cite{Pau97}.

By \cite[Theorem 1.2]{DHP23}, we can consider the relative $K_X-$MMP with respect to $\alpha$,
$$\begin{tikzcd}
X \arrow[r, "\alpha"] \arrow[d, "\beta", dotted] & A(X) \\
Y \arrow[ru, "\gamma"]                           &     
\end{tikzcd}$$
where the last step of $\beta$ is a Mori fiber space.
Notice that the restriction of the relative $K_X-$MMP with respect to $\alpha$ to any fiber $F$ is the $K_F-$MMP since $K_X|_F=(K_X+F)|_F=K_F$ by the adjunction formula.
But since any fiber $F$ of $\alpha$
is a compact K\"ahler threefold with nef anticanonical divisor such that $\tilde{q}(F)=0$ whose MMP consists of at most one step of a Mori fiber space,
$\beta$ is a morphism and $\beta$ consists of at most one step of a Mori fiber space.
In fact, since $\tilde{q}(F)=0$, $\pi_1(F)$ is finite by the structural theorem for compact K\"ahler threefolds with nef anticanonical divisor proven in \cite{MW23}.
Since $\pi_2(A(X))=0$, the homotopy exact sequence implies that $\pi_1(F)$ is a subgroup of $\pi_1(X)$ which is free abelian.
Thus $F$ is simply connected.
Therefore, $F$ is either a Calabi-Yau threefold, a product of $\P^1$ with a K3 surface or a rational connected threefold.
In the third case, $H^{2,0}(F)=H^{1,0}(F)=0$ and thus $H^{2,0}(X)=0$,
contradicting the assumption that $X$ is non-projective.
Thus the third possibility can be ruled out.
Note that the dimension of $Y$ is 4 or 3.

If $\dim(Y)=4$, any fiber $F$ of $\alpha$
is a simply connected Calabi-Yau threefold and $\beta$ is an isomorphism.
By Proposition \ref{CY} we have $c_1(X)=0$, 
contradicting the assumption that $K_X$ is not $\Q-$effective.

If $\dim(Y)=3$, the map $\gamma$ is a K3-fibration between compact K\"ahler manifolds.
Because $\dim(Y)=3$, $F$ is a product of $\P^1$ with a K3 surface and the restriction of $\beta$ to $F$ is the projection onto the K3 surface.
In particular, the set-theoretic fiber of $\gamma$ is alway a K3 surface.
Since $\beta$ is a $\Q-$conic bundle (see \cite{MW23}) and $X$ is smooth, $\beta$ is a conic bundle outside a codimension at least 2 closed analytic subset of $Y$.
The proof of \cite[Lemma 4.1]{HP15b} shows that $Y$ is klt and so in particular Cohen-Macaulay.
Since the fibers of $\gamma$ are K3 sufaces, $\gamma$ is flat by miracle flatness and hence smooth. It follows that $Y$ is smooth.

Again by Proposition \ref{CY}, up to some finite \'etale cover,
$Y$ is a product of $A(X)$ with a K3 surface.
Alternatively, we may argue as follows:
Since $-K_X$ is nef and $\beta$ is a conic bundle (i.e. $-K_X$ is $\beta-$ample) by (the proof of) Theorem 8.13 in \cite{DHP23}, $-K_Y$ is psef by \cite[Proposition 2.5]{MW23}.
On the other hand, $Y$ is not uniruled which implies that $K_Y$ is psef by \cite{Bru06}.
More precisely, the rational curves of $Y$ are contained in the fibers of $\gamma$ since $A(X)$ contains no rational curves.
Thus, if $Y$ were uniruled  the fibers of $\gamma$ would be uniruled, too, contradicting the fact that the fibers of $\gamma$ are K3 surfaces.
Thus $c_1(Y)=\frac{1}{4} \beta_*(c_1(-K_X)^2)=0$
by conic bundle formula \cite[Proposition 3.5]{MW23}.

By the Fischer-Grauert theorem \cite{FG65}, $\beta$ is  locally trivial since $\P^1$ is rigid.
In particular, there exists a line bundle $L$ on $X$ such that $-K_X=L^{\otimes 2}$ (see e.g.  \cite[Lemma 6]{ARM14}).
Since $L$ is $\beta-$very ample, $X=\P(\beta_* L)$.
By \cite[Claim 2.8]{MW23}, $\beta_* L$ is weakly positively curved.
On the other hand,
$$c_1(\beta_*L)=\frac{1}{4} \beta_*(c_1(-K_X)^2)=0. $$
By \cite[Main theorem]{numer_flat}, $\beta_* L$ is numerically flat.
Since a K3 surface is simply connected, $\beta_* L$ is the pull back of some numerically flat vector bundle over the elliptic curve by the next Lemma.
\end{proof}
\begin{lem}\label{Lemma 1}
Let $X,Y$ be two compact K\"ahler manifolds.
Assume that $\pi_1(Y)=0$.
Let $E$ be a numerically flat vector bundle over $X \times Y$.
Then there exists a numerically flat vector bundle $F$ on $X$ such that $E=p_1^* F$ where $p_1: X \times Y \to X$ is the natural projection.
\end{lem}
\begin{proof}
By \cite[Theorem 1.18]{DPS94}, $E$ is a successive extension of hermitian flat vector bundles on $X \times Y$.
Since any hermitian flat vector bundle corresponds to a representation of $\pi_1(X \times Y) \simeq \pi_1(X)$,
these hermitian flat vector bundles are pullbacks of some hermitian flat vector bundles on $X$.

Now consider an extension of pullbacks of hermitian flat vector bundles $E_1, E_2$ on $X$
which corresponds to an element of
$H^1(X \times Y, p_1^* \hom(E_2, E_1))$, and thus an element of $H^1(X, \hom(E_2, E_1)) \otimes H^0(Y, \cO_Y) \oplus H^0(X,  \hom(E_2, E_1)) \otimes H^1(Y, \cO_Y)$  by the K\"unneth formula which gives $H^1(X, \hom(E_2, E_1))$.
In particular, the extension is the pullback of some extension of the hermitian flat vector bundles on $X$.
Induction on the rank of $E$ finishes the proof. 
\end{proof}

\section{First Case: Fibers are Calabi-Yau}

\begin{prop}\label{CY}
Let $X$ be a compact K\"ahler manifold and $f: X \to T$ be a smooth morphism onto a complex torus.
Assume that each fiber of $f$ is a simply connected Calabi-Yau manifold.
Then $c_1(X)=0$.
\end{prop}
\begin{proof}
We start by recalling the following result of \cite[Remark 3.2]{TZ14}:

\textit{
If $g : X \to Y$ is a holomorphic submersion between compact K\"ahler manifolds with fibers $X_y$ having trivial canonical bundle and base $Y$ which admits
a K\"ahler metric with non-negative Ricci curvature, then $g$ is a holomorphic
fiber bundle.
}

This result says that $f$ is locally trivial.
In other words, we have no deformation of complex structure among the fibers of $f$ which in turn implies that the Weil-Petersson metric on $T$ vanishes identically.
To conclude, we use the fact that for a smooth fibration of Calabi-Yau manifolds, there is a close relation between the variation of complex structure and the positivity of the relative canonical bundle. To this end, we consider the fiberwise unique Ricci-flat K\"ahler-Einstein metrics $\omega_s$ in the corresponding K\"ahler classes, that are given by the polarization. The relative volume form $\omega_{X/T}^{n}=gdV$, where $\dim X_s=n$, induces a hermitian metric $g^{-1}$ on the relative canonical bundle $K_{X/T}=K_X$. Its curvature is given, up to a numerical constant, by the pullback of the Weil-Petersson form downstairs, i.e. the $L^2$ inner product of harmonic representatives of the Kodaira-Spencer classes, see e.g. \cite[Eqn. 1.1]{BCS20}:
$$
2\pi c_1(K_X, g^{-1}) = \sqrt{-1}\partial \overline{\partial} \log g = \frac{1}{\operatorname{vol} X_s} f^*\omega^{WP}.
$$ 
But by the isotriviality, the Weil-Petersson form $\omega^{WP}$ vanishes identically, hence \\ $c_1(K_X, g^{-1})=0$.
\end{proof}
Another proof can be given by moduli of K\"ahler polarized Calabi-Yau manifolds:
\begin{proof}
First we want to show that the fibration is isotrivial. In the projective case, this follows from \cite[Theorem A]{Den19}. The same is true in our situation:  By the work of Fujiki and Schumacher, see \cite[Thm. 5.4]{FS90} and also \cite{Sch83, Fu84}, there exists a coarse moduli space of K\"ahler polarized K\"ahler manifolds with $c_1=0$, denoted by $\mathfrak{M}$, which has the structure of a complex orbifold. In analogy to the canonically polarized case, by the work of \cite{Bra15} or \cite{TY18}, this moduli space carries an orbifold Finsler metric  whose sectional curvature can be bounded by a negative constant. Hence, the moduli of K\"ahler polarized Calabi-Yau manifolds is Kobayashi hyperbolic in the orbifold sense, cf. \cite[Thm. C]{Den19}.  

Now we can prove the isotriviality as in \cite[pf. of Thm. C]{Den19}: Suppose that $f$ is not isotrivial. The K\"ahler form $\omega_X$ induces a polarization $\lambda_{X/T} \in R^1f_*\Omega_{X/T}(T)$, cf. \cite{FS90}. Hence we obtain a moduli map $\varphi: T \to \mathfrak{M}$ which is not constant. By construction, 
$\mathfrak{M}$ is locally the quotient of Kuranishi spaces, so the map $\varphi$ is locally liftable and thus an orbifold morphism. For a pair of points $p,q \in T$ such that $\varphi(p) \neq \varphi(q)$, we obtain by the distance decreasing property of the Kobayashi pseudo-metric, see \cite[Lemma 3.4]{Den19}, that $d_Y(p,q) \geq d_{\mathfrak{M}}(\varphi(p), \varphi(q))>0$, which contradicts the fact that the complex torus $T$ is hyperbolically special, i.e. the Kobayashi pseudo distance vanishes identically on $T$. The rest of proof follows from that of the previous one.
\end{proof}
Following the arguments of \cite[Theorem 3.1]{TZ14}, we have the following result:
\begin{prop} \label{Prop3}
Let $X$ be a compact K\"ahler non-uniruled manifold of dimension $n$ such that 
$\tilde{q}(X)=q(X)$, the general fiber is a Calabi-Yau manifold and 
$-K_X$ is nef.
Then we have $c_1(X)=0$.
\end{prop}
\begin{proof}
Let $\alpha: X \to A(X)$ be the Albanese morphism which is smooth in codimension 1 by \cite{Cao13}.
If $\dim (A(X))=n$, by \cite{Cao13}, $X$ is a modification of $A(X)$.
In particular, we have equality of Kodaira dimensions $\kappa(X)=\kappa(A(X))=0$ and thus $K_X$ is psef.
Since $-K_X$ is nef, $c_1(X)=0$.
Beauville-Bogomolov decomposition implies that in fact $X \simeq A(X)$.
Therefore, we only consider the case where $\dim(A(X))<n$.
Without loss of generality, we may assume also that $q(X) \geq 1$.

Let $Z$ be a codimension at least two closed analytic subset of $A(X)$ such that $\alpha: \alpha^{-1}(A(X) \setminus Z) \to A(X) \setminus Z$ is smooth.

Consider the universal cover $\pi: \C^{2\dim(A(X))} \to A(X)$.
By the codimension condition of $Z$,
$\pi: \pi^{-1}(A(X) \setminus Z) \to A(X) \setminus Z$ is the universal cover of $A(X) \setminus Z$.

Endow the local system $R^k\alpha_*\R$ associated to $\alpha: \alpha^{-1}(A(X) \setminus Z) \to A(X)\setminus Z$ with the geometric real variation of polarized Hodge structures, where $k=n-\dim(A(X))$, cf. \cite{TZ14}.
The construction of period domains and period mappings for real polarized Hodge structures by Griffiths gives a
well-defined holomorphic period map $p: \pi^{-1}(A(X) \setminus Z) \to \mathcal{D}$ where $\mathcal{D}$ is the associated period domain. 
For example, the period domain is the upper half plane $\mathbb{H}$ if the fiber $F$ is an elliptic curve.
In the case $\dim(A(X))=n-2$, as observed on page 305 of \cite{Kob90}, the period domain of real polarized K3 surfaces is the non-compact Hermitian symmetric domain $O(2,19)/SO(2) \times O(1,19)$.
By Harish-Chandra's embedding theorem, the period domain can be realized as an open subset of a complex vector space.
By Hartogs theorem, the period map $p$ extends to a holomorphic function on $\C^{2\dim(A(X))}$.
In the general case, the extension follows from \cite[Corollary 9.5]{GS69} based on the geodesic completeness of the period domain and the distance decreasing property of morphisms between manifolds whose curvature have strict negative upper bounds (which implies the continuous hence holomorphic extension).

The period map is thus constant by Liouville's theorem in the case $\dim(A(X))=n-1$.
In the general case,
it follows from Schwarz lemma \cite[Lemma 3.3]{TZ14} and the existence of a hermitian metric on $\mathcal{D}$ whose holomorphic sectional curvature is bounded form above by a strictly negative constant along the horizontal subbundle of $T \mathcal{D}$ defined by Griffiths \cite{Gri70}.

But now the differential of the period map $p$ at $z$ is a composition of the Kodaira-Spencer map $\rho_z: T_z (\pi^{-1}(A(X) \setminus Z)) \to H^1(\alpha^{-1}(\pi(z)), T_{\alpha^{-1}(\pi(z))})$, and the cup product map $H^1(\alpha^{-1}(\pi(z)), T_{\alpha^{-1}(\pi(z))}) \to T_{p(z)}\mathcal{D}$. Since the second map is injective by infinitesimal Torelli,
the Kodaira-Spencer map of the pullback family over $\pi^{-1}(A(X) \setminus Z)$ is identically zero.

By \cite[Theorem 4.6]{Kod05}, since $h^1(F, T_F)=h^{1,n-1}(F)$ which is constant for the K\"ahler fibration, $\alpha$ is locally trivial over $A(X) \setminus Z$.
No deformation of complex structure among the fibers of $\alpha$ implies that the Weil-Petersson metric on $A(X) \setminus Z$ is identically zero.
Again by the result of \cite{BCS20}, the pullback of the Weil-Petersson metric represents the first Chern class of relative  canonical divisor $K_{(X \setminus \alpha^{-1}(Z))/(A(X) \setminus Z)}$ modulo a constant factor, which implies that the first Chern class of $X \setminus \alpha^{-1}(Z)$ is trivial.

The preimage of the non-flat locus $W$ of $\alpha$ is codimension at least 2 by Lemma \ref{FlatLocus}. Since the dimension of the fibers is the expected one for a point on the base where the map is flat, the preimage of the non-smooth locus of $\alpha$ is also codimension at least 2.
Thus, by Lemma \ref{Cohomology}, $H^2(X, \C) \simeq H^2(X \setminus \alpha^{-1}(Z), \C )$ and hence $c_1(X)=0$.
\end{proof}

\begin{lem}\label{Cohomology}
Let $X$ be complex manifold of dimension $n$ containing a closed analytic subset $W$ of codimension at least $k$.
Then we have for $i \leq 2(k-1)$,
$$ H^{i}(X, \C) \simeq H^{i}(X \setminus W,\C).$$
\end{lem}
\begin{proof}
By Poincar\'e duality, it is equivalent to show that for $i \leq 2(k-1)$,
$$ H^{2n-i}_c(X,\C) \simeq H^{2n-i}_c(X \setminus W,\C)$$
which follows from long exact sequence
$$ H^{2n-i-1}_c(W,\C)=0 \to H^{2n-i}_c(X \setminus W,\C) \to H^{2n-i}_c(X,\C) \to H^{2n-i}_c (W,\C)=0.$$
Here, the vanishing of cohomologies follows from dimension reasons.
\end{proof}

\begin{lem}\label{FlatLocus}
Let $X$ be a compact K\"ahler manifold of dimension $n$ such that $-K_X$ is nef.
Let $Y$ be a compact K\"ahler manifold such that $c_1(Y)=0$ with a holomorphic map $\alpha: X \to Y$.
Assume that the non-flat locus $W \subset Y$ is of codimension at least 2.
Then $\alpha^{-1}(W)$ is of codimension at least 2.
\end{lem}
\begin{proof}
We follow the arguments of  \cite[Lemma 3.2]{CH19}.
Let $\omega_X$ be a K\"ahler form on $X$.
By Hironaka's flattening theorem \cite{Hir75}, we have a diagram
$$
\begin{tikzcd}
\tilde{X} \arrow[rd, "\beta"] \arrow[r] & X_1 \arrow[r] \arrow[d, "\tilde{\alpha}"] & X \arrow[d, "\alpha"] \\
                                        & \tilde{Y} \arrow[r, "\tau"]            & Y                 
\end{tikzcd}
$$
such that $\pi$ which is the composition $\tilde{X} \to X_1 \to X$, $X_1 \to X$, $\tau$ are modifications, $\tilde{X}, \tilde{Y}$ are smooth and $\tilde{\alpha}$ is flat.
Denote
$$K_{\tilde{X}}=\pi^* K_X+E, K_{\tilde{Y}}=F$$
where $E,F$ effective divisors.
Let $\omega_{\tilde{X}}$ be a K\"ahler form on $\tilde{X}$.
For any $\epsilon >0$,
$\pi^*(-K_X)+\epsilon \omega_{\tilde{X}}$ is K\"ahler.
Since $\beta(E) \neq \tilde{Y}$,
$K_{\tilde{X}}+\pi^*(-K_X)+\epsilon \omega_{\tilde{X}}$
is K\"ahler when restricting to a general fiber of $\beta$.
By \cite[Theorem 1.1]{Pau12},
$K_{\tilde{X}/\tilde{Y}}+\pi^*(-K_X)+\epsilon \omega_{\tilde{X}}$
is psef.
Thus after passing to a limit,
$E-\beta^* F$ is psef and we get
$$(E-\beta^* F, \pi^* (\omega)^{n-1})=-(\beta^* F, \pi^* (\omega)^{n-1}) \geq 0.$$
Therefore, $\beta^*(F)$ must be $\pi-$exceptional and,
in particular, $\alpha^{-1}(W)$ is of codimension at least 2. 
\end{proof}
As a direct application, we have the following corollary:
\begin{cor}
Let $X$ be a compact K\"ahler manifold such that 
$\tilde{q}(X)=q(X)=2$ and
$-K_X$ is nef.
Then the Albanese map of $X$ is flat.
\end{cor}

\begin{rem}
As conjectured in \cite{BDPP}, if $X$ is not uniruled, $K_X$ should be psef.
If $-K_X$ is nef, we should have $c_1(X)=0$.
The interest of the above arguments is a substitute of \cite{BDPP} to prove the vanishing of Chern class.

For example, let $X$ be a compact K\"ahler non-uniruled $n-$fold with nef anticanonical divisor.
If $q(X)=\tilde{q}(X) \geq n-3$, by adjunction formula, the anticanonical divisor of a general fiber $F$ is nef.
We claim that the general fiber is not uniruled.
If $F$ is uniruled (which is meaningful by \cite{Fu81}),
there exists a rational curve passing through a general point of $X$ which implies that $X$ is uniruled which is a contradiction.
By the condition on the dimension, $K_F$ is also psef which implies that the fibers are Calabi-Yau.
Proposition \ref{Prop3} then implies that $c_1(X)=0$.

Combined with Proposition \ref{Prop1}, the only unknown cases of the conjecture of \cite{BDPP} for compact K\"ahler fourfold with nef anticanonical divisor are simply connected ones.
Note that it is not enough to apply \cite{Bru06} to conclude in this case.
Conjecturally, if the manifold $X$ is not projective, $h^{2,0}(X) \neq 0$ and $X$ should be hyperk\"ahler or a product of K3 surfaces.
But there exists a hyperk\"ahler fourfold with trivial $H^{3,0}(X)$ which forbids the application of \cite{Bru06}.
\end{rem}

\section{Second case: Fibers are projective spaces}

\begin{lem}\label{Lem4.1}
Let $X$ be the complement of a closed analytic subset of codimension at least 2 in a complex torus.
Then any $\P^{r-1}-$bundle over $X$ is isomorphic to $\P(E)$ for some holomorphic vector bundle $E$ over $X$ with $\det(E)=\cO_X$ up to some finite \'etale cover.
\end{lem}
\begin{proof}
Denote $X=A\setminus Z$ where $A$ is a torus and $Z$ is of codimension at least 2.
The obstruction of a $\P^{r-1}-$bundle over $X$ being isomorphic to $\P(E)$ for some holomorphic vector bundle $E$ over $X$ with $\det(E)=\cO_X$ lies in $H^2(A \setminus Z, \Z_r)$ by the theory of Brauer groups.

Here, we
follow the proof of \cite[Lemma 7.4]{CP91}.
From the exact sequence
$$0 \to \Z_r \to \mathrm{SL}(r, \cO_X) \to \mathrm{PSL} (r, \cO_X) \to 0, $$
we see that the obstruction lies in
$$H^2(A \setminus Z, \Z_r)=H^2(A, \Z_r)=H^2(A, \Z)/H^2(A, r\Z)$$
by the following Lemma \ref{Lem4.2}.
Thus the obstruction vanishes after passing to some finite \'etale cover.
Notice that any finite \'etale cover of $A \setminus Z$ extends uniquely to a finite \'etale cover of $A$.
This finishes the proof of the lemma.
We refer to \cite{El82} for more information about Brauer groups.
\end{proof}
\begin{lem}\label{Lem4.2}
Let $X$ be the complement of a closed analytic subset of codimension at least $k$ in a complex manifold $Y$.
Then for any $i \leq 2k-1$, $H^i(X, \Z) \simeq H^i(Y,\Z)$.
Similar results hold if we change the coefficient ring to $\Z_r$ for any $r \geq 2$.
\end{lem}
\begin{proof}
By stratification of closed analytic subset (cf. e.g. \cite[Chap II, Proposition 5.6]{agbook}), without loss of generality, we may assume that $X$ is the complement of a smooth complex submanifold of codimension at least $k$ in a complex manifold $Y$.
Consider long exact sequence
$$H^{i-1}(Y, X, \Z) \to H^i(Y, \Z) \to H^i(X, \Z) \to H^i(Y, X, \Z).$$
It is enough to show that for $i \leq 2k-1$, by considering a tubular neighborhood of $Y \setminus X$ and the excision theorem,
$$H^i(Y, X, \Z) \simeq H^i( Y \setminus X \times  \mathbb{D},  Y \setminus X \times  \mathbb{D}^*,\Z)=0$$
where $\mathbb{D}$ is the disk of dimension of the real codimension of $Y \setminus X$ and $\mathbb{D}^*$ is the punctured disk.
Since $H^\bullet(\mathbb{D}, \mathbb{D}^*,\Z)$ is finite generated free group, K\"unneth theorem implies that
$$H^i( Y \setminus X \times  \mathbb{D},  Y \setminus X \times  \mathbb{D}^*,\Z) \simeq \oplus_k H^{i-k}(Y \setminus X, \Z) \otimes H^k(\mathbb{D}, \mathbb{D}^*, \Z)=0$$
by the codimension condition.
\end{proof}
\begin{rem}
In general, the obstruction of a $\P^{r-1}-$bundle being the projectivization of some holomorphic vector bundle lies in the torsion part of $H^2(X, \cO_X^*)$.
In \cite{PS98}, the authors show its vanishing in the setting of Lemma \ref{Lem4.1} when the dimension of $X$ is two by using a result of Malgrange, the remark following \cite[Thm. 8]{Mal55}.
 More precisely, they show that  $H^2(X, \cO_X)=0$ which implies that  the torsion part of $H^2(X, \cO_X^*)$ is the same as  the torsion part of $H^3(X, \Z)$.
 Note that the latter is trivial by Lemma \ref{Lem4.2}.
 However, in general, their arguments fail in arbitrary dimension at the point that
 $H^2(X, \cO_X)$ can be non-trivial.

 To give an easy and concrete example, let $X$ be the complement of $A_1 \times \{p\}$ in $A_1 \times A_2$ where $A_1,A_2$ are complex tori with $A_1$ of dimension at least 2.
 Consider the restriction maps
 $$H^{2}(A_1 \times A_2, \cO_{A_1 \times A_2})  \to H^2(X,\cO_X) \to H^2(A_1 \times \{q\},\cO_{A_1 \times \{q\}}) $$
 where $q \neq p \in A_2$.
By the K\"unneth formula, the composition map is surjective with non-trivial image.
In particular, $H^2(X, \cO_X) \neq 0$.
\end{rem}

We give a more topological proof of the following result in \cite{CH17}.
\begin{thm}\label{Proj}
Let $X$ be a compact K\"ahler manifold such that 
$-K_X$ is nef.
Assume that the general fiber of the Albanese map is $\P^{r-1}$.
Then we have $X \simeq \P(E)$ for some numerical flat vector bundle of rank $r$ over $A(X)$ up to some finite \'etale cover.
\end{thm}
\begin{proof}
By Lemma \ref{FlatLocus} and \cite{ARM14}, since the Albanese map is smooth in codimension one with $\P^{r-1}$ as the general fiber, there exists a closed analytic subset $Z$ of codimension at least 3 in $A(X)$  (non-flat locus of $\alpha$) such that $X$ is a $\P^{r-1}-$bundle over $A(X) \setminus Z$.
By Lemma \ref{Lem4.1}, there exists a holomorphic vector bundle $E$ over $A(X) \setminus Z$ of rank $r$ with $\det(E)=\cO_X$ such that
$$X|_{\alpha^{-1}(A(X) \setminus Z)} \simeq \P(E)$$
up to passing to some finite \'etale cover of $X$.

In particular, up to some finite \'etale cover, $c_1(E)=0$ and $-K_X=\cO_{X}(r)$ on $A(X) \setminus Z$.
By the Leray-Hirsch theorem,
 we have
$$H^{2}(\alpha^{-1}(A(X) \setminus Z), \R) \simeq \R c_1(\cO_{\P(E)}(1)) \oplus \alpha^* H^{2}(A(X) \setminus Z, \R).$$
Let $\omega$ be a normalised K\"ahler form on $X$  (i.e. such that the restriction of $\omega+c_1(K_X)$ to the general fiber of $\alpha$ is trivial, hence for all fibers over $A(X) \setminus Z$ by the Leray-Hirsch theorem).
We claim that
there exists $\beta \in H^{1,1}(A(X), \R)$ such that we have 
$$c_1(K_X)+ \omega=\alpha^* \beta$$
in $H^{1,1}(X, \R)$.
By Lemma \ref{Cohomology}, it is enough to show the existence of $\beta \in H^{1,1}(A(X), \R)$ satisfying the equality on $\alpha^{-1}(A(X) \setminus Z)$.
Define
$$\beta:= \frac{1}{r^{r-1}} \alpha_*(c_1(-K_X)^{r-1} \cdot \omega).$$
The claim follows by construction.
(The same equality holds in more finer cohomology as integral Bott-Chern cohomology as shown in \cite{Wu19a}.)

In particular, $-K_X$ is $\alpha-$ample and $\alpha$ is projective.
We claim that $E$ can be extended over $A(X)$ as a reflexive sheaf.
The line bundle $\cO_{\P(E)}(1)$ can also be endowed with a smooth metric whose curvature has finite $L^2$ norm and whose local potentials
are the local potentials of $-K_{X}$ divided by $r$.
By the removable singularity theorem
(see \cite[Lemma 1]{BS94}),
the line bundle $\cO_{\P(E)}(1)$ has a reflexive extension. Its direct image, denoted by $\cF$,  is a reflexive extension of $E$ by \cite[Cor.1.7]{Har80}. (Note that the proof of \cite{Har80} goes through as long as the preimage of the non-flat locus is of codimension at least two.)
Note that the extended line bundle of $
\cO_{\P(E)}(1)$ is nef and $\alpha-$ample as $-K_X$.
By the positivity of direct images \cite[Proposition 2.6]{MW23},
$\cF$ is weakly positively curved.
Since $c_1(E)=0$, $c_1(\cF)=0$ on $A(X)$.

Next, we show that $\cF$ is locally free by \cite{BS94}.
By nefness of the extended line bundle of $\cO_{\P(E)}(1)$ and the fact that $Z$ is of codimension at least 3 (such that $H^4(A(X) \setminus Z, \R) \simeq H^4(A(X), \R)$),
$c_1^2(\cF)-c_2(\cF)=-c_2(\cF)$ contains a closed positive current whose restriction on $A(X) \setminus Z$ represents $\alpha_*(c_1(\cO_{\P(E)}(1))^{r+1})$.
We refer to \cite[Remark 1]{Wu23} for the construction of Chern classes of a coherent sheaf over $A(X) \setminus Z$ in the de Rham cohomologies.
If $\cF$ is stable, by \cite[Corollary 3]{BS94}, the equality case of the Bogomolov-Gieseker inequality implies that $\cF$ is a projectively flat vector bundle.
Since $c_1(\cF)=0$, $\cF$ is hermitian flat.

We show in the following that $\cF$ is a numerically flat vector bundle following the arguments of \cite[Proposition 2.7]{CCM21}, if $\cF$ is not stable.
If $\cF$ is not stable, consider the Harder-Narasimhan filtration of $\cF$
respect to $\omega$, say
$$\cF_0=0 \to \cF_1 \to \cF_{2} \to \cdots\to \cF_m:=\cF$$
where $\cF_i /\cF_{i-1}$ is $\omega$-stable for every $i$ and $\mu_1 \geq \mu_2 \geq \cdots \geq \mu_m$, and where $\mu_j=\mu_\omega(\cF_j / \cF_{j-1})$ is the slope of $\cF_j / \cF_{j-1}$ with respect to $\omega$.
Now, consider the coherent subsheaf $\cS=\cF_{m-1}$.
Without loss of generality, we can assume that $\cS$ is reflexive by taking the double dual if necessary, as this preserves the rank, first Chern class and slope. Then we get a short exact sequence
$$0 \to \cS \to \cF \to \cQ \to 0.$$

As a quotient of $\cF$, $\cQ$ moduo its torsion part is also weakly positively curved.
In particular, $c_1(\cQ)$ is pseudoeffective.
Together with the stability condition of the filtration and the fact that
$$0=c_1(\cF)=c_1(\cS)+c_1(\cQ),$$
we can easily
check that $c_1(\cS)=c_1(\cQ)=0$.

By \cite[Corollary 1.19]{DPS94}, the inclusion $\cS \to \cF$ is a bundle morphism on the locally free locus of $\cF$  (i.e. $A(X) \setminus Z$).
In particular, $\cQ$ is locally free on $A(X) \setminus Z$.
Compared to the case of an arbitrary weakly positive reflexive sheaf with vanishing first Chern class over a projective manifold as studied in \cite{CCM21}, there exist smooth metrics with lower bounds on the curvature of $\cO_{\P(E)}(1)$ (induced from $-K_X$ and the nefness condition of $-K_X$).
Now we deviate from \cite{CCM21} and give a more simplified argument based on these metrics.
Restriction of these metrics on $\cO_{\P(\cQ)}(1)$ to $A(X) \setminus Z$
shows that over  $A(X) \setminus Z$, the cohomology class
$$\alpha_*(c_1(\cO_{\P(\cQ)}(1))^{\rank{\cQ}+1})=c_1^2(\cQ)-c_2(\cQ)=-c_2(\cQ)$$
contains a closed positive $(2,2)-$current.
Since $Z$ is of dimension at least 3,
this current extends across $Z$.
In particular, $-c_2(\cQ^{**})$ contains a closed postive current.
On the other hand, $\cQ^{**}$ is stable.
The equality case of the Bogomolov-Gieseker inequality is attained and thus $\cQ^{**}$ is also a hermitian flat vector bundle by \cite{BS94}.

(If the rank $r=2$, $\cQ$ can only be of rank 1.
We have that $c_1(\cQ^{**})=c_1(\cQ)=0$ and $\cQ^{**}$ is a hermitian flat line bundle.)

The natural morphism on the locally free locus of $\cF$
$$\wedge^{\rank{\cF}-\rank{\cQ}-1} \cF \otimes \det(\cQ)^{-1} \to \cS$$
implies that $\cS$ is weakly positively curved after extension to $A(X)$ and it admits a sequence of smooth metrics with increasing lower bound on $A(X) \setminus Z$.
Since
$$c_2(\cF)=c_2(\cQ^{**})+c_2(\cS)+c_1(\cS) \cdot c_1(\cQ)=c_2(\cS),$$
if $\cS$ is not stable, we can continue to think about its Harder-Narasimhan filtration and by induction on the rank conclude that $\cS$ is a numerically flat vector bundle.

By \cite[Lemma 4]{numer_flat}, the restriction map
$$H^1(A(X), \hom(\cQ^{**}, \cS)) \to H^1(A(X) \setminus Z, \hom(\cQ^{**}, \cS))$$
is surjective.
Hence the extension class in $H^1(A(X) \setminus Z, \hom(\cQ^{**}, \cS))$ obtained from the Harder-Narasimhan filtration on $A(X) \setminus Z$ can be extended to an
extension class in $H^1(A(X) \setminus Z, \hom(\cQ^{**}, \cS))$. The extended class determines a vector bundle
whose restriction to $A(X) \setminus Z$ is isomorphic to $\cF$ by construction. It follows that this vector
bundle coincides with $\cF$ on $A(X)$ since $\cF$ is reflexive.
Note that $\cF$ is a numerically flat vector bundle in this case.

In conclusion, we have shown that $X$ is bimeromorphic to $\P(\cF)$ over $A(X)$ which is isomorphic in codimension one.
Following \cite[Lemma 6.39]{KM98}, we show that it is in fact an isomorphism.
Let $p: \P(\cF) \to A(X)$ be the natural projection.
It is enough to show that for any $\alpha-$ample Cartier divisor $D$ on $X$, the induced Cartier divisor $D'$ on $\P(\cF)$ is $p-$nef.
Since $p$ is a projective morphism, it is enough to show for any (compact) curve $C$ contained in a fiber of $p$, the intersection number $(c_1(D') \cdot C) \geq 0$.
On the other hand, since $p$ is locally trivial, $C$ deforms to a curve over $A(X) \setminus Z$ which is thus isomorphic to a curve in $X$.
Thus $(c_1(D') \cdot C) \geq 0$ since $D$ is $\alpha-$ample.
\end{proof}
\begin{rem}
Notice that manifolds fulfilling the requirements of Theorem \ref{Proj} (or compact K\"ahler manifolds whose finite \'etale cover is such a manifold) can be algebraically approximated.
Indeed, the tangent bundle of $\P(E)$ of a numerically flat vector bundle over a torus is nef by \cite[Lemma 7.4]{CP91} (the proof works identically in the analytic setting).
Then the algebraic approximation exists by \cite[Theorem 6.3]{Cao15}.
In fact, by Simpson's correspondance \cite{Sim92}, $E$ is flat which corresponds to a representation of $\pi_1(X)$ in 
$\mathrm{GL}(r,\C)$ where $r$ is the rank of $E$.
Since the fundamental group is preserved under deformation, we can deform $\P(E)$ when deforming the torus to get the algebraic approximation.
\end{rem}
\begin{rem}\label{remark 3}
As shown in \cite[Example 1.2]{Cat83}, Hirzebruch surfaces can admit non-trivial (flat) deformations over a polydisc of dimension at least 2 which are locally trivial outside the origin.
In particular, the analogous result of \cite{ARM14} does not hold for Hirzebruch surfaces.
By \cite[Theorem 1.3]{BC18}, the only globally rigid rational surface is $\P^2$.
(Here $\P^2$ being globally rigid means that any surface deformation equivalent to $\P^2$ is $\P^2$.)
It seems that the above topological approach (by extension from a big Zariski open set, i.e. the complement in the Albanese torus is of codimension at least 2) is hard to generalise to a Del Pezzo fibration. 
In particular, it seems unclear how to classify uniruled compact K\"ahler fourfolds with nef anticanonical line bundle and two-dimensional Albanese torus.
(Notice that unlike to the pseudoeffectivity of the anticanonical line bundle, the example in \cite{Cat83} also shows that the nefness of the anticanonical line bundle is not a closed condition.)
\end{rem}

The above phenomenon in remark \ref{remark 3} can be ruled out if each fiber of the deformation has nef anticanonical line bundle.
In particular, assuming that the Albanese morphism is smooth and fibers are surfaces, we get a structure theorem, cf. \cite[Cor. 3.4]{PS98}, 
as we show in the next section.

\section{The general uniruled case when fibers are surfaces}

We start with a K\"ahler version of \cite[Proposition 0.4]{PS98}:

\begin{prop}
\label{PS04}
 Let $X, Y$ be compact K\"ahler manifolds, $\pi: X \to Y$ a smooth
morphism with $-K_F$ is nef for all fibers $F$ which are moreover non-minimal surfaces.
Assume that there exists a $(-1)$-curve $C$ in some fiber $F$ and a K\"ahler class $\omega$ on $X$ such that
$$(K_F+C) \cdot \omega<0.$$
Then there exists an \'etale base change $\sigma : X' = X \times_Y Y' \to Y'$
induced by an \'etale map $Y' \to Y,$ and a smooth effective divisor $S \subset X'$ such that the restriction
$ \sigma \vert_S : S \to D$ yields a $\P_1$-bundle structure on $S$, and $S \cap F$ is a
$(-1)$-curve in $F$ for all fibers $F.$ Hence $X'$ can be blown down along $\sigma \vert_S.$

\end{prop}

\begin{proof} For convenience of the reader, we reproduce the arguments from \cite[Proposition 0.4]{PS98}.
First, note that all non-minimal surfaces $F$ with $-K_F$ nef are isomorphic to the
plane $\P_2$ blown up in at most 9 points in sufficiently general position.

We fix the K\"ahler class $\omega$ on $X$. Consider the fiber $F$ of $\pi$ in the assumption and
take a $(-1)$-curve $E \subset F$ such that $\omega \cdot E$ is minimal under all $(-1)$-curves in $F.$
By assumption,
$$(K_F+E) \cdot \omega <0.$$

The normal bundle of $E$ satisifes
$$0 \to N_{E/F} \to N_{E/X} \to N_{F/X}|_E \simeq \cO_E^{\dim Y} \to 0.$$
Since $H^1(E, \mathrm{Hom}(N_{F/X}|_E,  N_{E/F}))=H^1(\P^1, \cO(-1))=0$, the short exact sequence splits and the normal bundle
is of the form
$$ N_{E \vert X} = \cO^{\dim Y} \oplus \cO(-1).$$
In particular, $E$ corresponds to a regular point of the Douady space. 
Note that the irreducible component of the Douady space containing $E$ is of dimension of $Y$ which is moreover compact.
We will show in the following that it is in fact smooth.

To fix notations,
there exists a compact complex space $Y'$ and an irreducible effective divisor
$M \subset X \times Y'$, flat over $Y'$ such that $M \cap (X \times 0) = E,$ identifying  $X \times 0$ with $X.$
We let
$$ E_t = M \cap (X \times t)$$
and shall identify $X$ with $X \times t.$

{\it Claim 1} Every $E_t$ is supported in some fiber $F'$ of $\pi.$

Otherwise, $\pi(E_t)$ will be a non-trivial effective curve.
This contradicts the fact that the $E_t$ are in the same cohomology class and so are the $\pi(E_t)$.
Note that $[\pi(E)]=0 \in H^{2 \dim (Y)-2}(Y, \Z)$.
In particular, there exists commuting diagram
$$\begin{tikzcd}
M \arrow[r] \arrow[d] & X \arrow[d, "\pi"] \\
Y' \arrow[r, "\varphi"]          & Y.
\end{tikzcd}$$
The holomorphic map $\varphi: Y' \to Y$ is surjective.
More precisely,  if it is not surjective, the image of $Y'$ is a proper analytic subset of $Y$.
In particular, by semicontinuity of the dimension of the fiber, the preimage of $\pi(E)$ has strictly positive dimension.
In other words, $F$ contains a family of rational curves containing $E$ of strictly positive dimension.
This contradicts the fact that $E$ is rigid in $F$.

To show that the morphism $Y' \to Y$ is in fact locally isomorphic near the point coresponding to $E$, we need a local description of the Douady space using Kodaira's theory, cf. \cite{Kod62}.
Locally, one can cover $X$ by coordinate charts $(z_i^\lambda,w_j^\lambda ), i=1,2; 1 \leq j \leq \dim(Y)$ where $\lambda$ is the index of the coordinate chart such that $\pi(z_i^\lambda,w_j^\lambda)=w_j^\lambda$ and $E$ is locally given by $\{z_2^\lambda= w_j^\lambda=0, \forall j \}$.
By Kodaira's theory, the deformation of $E$ is locally given by $(z_1^\lambda, z_2^\lambda=\psi^\lambda(z_1^\lambda, t), w_j^\lambda=\varphi_j^\lambda(z_1^\lambda, t) )$
for some local holomorphic functions $\psi^\lambda(z_1^\lambda, t), \varphi_j^\lambda(z_1^\lambda, t)$ where $t$ is the coordinate of the parameter space.
The differential of the map $t \mapsto (\psi^\lambda(z_1^\lambda, t), \varphi_j^\lambda(z_1^\lambda, t))$ gives an identification of the tangent space of the parameter space at 0 with $H^0(E, N_{E/X})$.
In particular, the differential of the map $t \mapsto ( \varphi_j^\lambda(z_1^\lambda, t))$ is isomorphic.
(And $\psi^\lambda(z_1^\lambda, t)$ is independent of $t$.)
On the other hand, in local coordinates,
$\varphi(t)=(\varphi_j^\lambda(z_1^\lambda, t))$
which implies that $\varphi$ is locally isomorphic.

Note that the above calculation is valid for any smooth $(-1)$-rational curve.

Observe next that
$$ -K_X \cdot E_t = -K_X \cdot E = 1.$$
Consider some $E_{t_0}$.
Write
$$ E_{t_0} = \sum_{i \geq 0} a_i C_i $$
with irreducible curves $C_i$ (as a cycle, neglecting the possible embedded point). Since $-K_{F'}$ is nef, we conclude (after renumbering possibly) that
$$a_0 = 1, -K_{F'}\cdot C_0 = 1$$
and that
$$-K_{F'}\cdot C_i = 0, i \geq 1.$$

Since $E_{t_0}$ is a deformation of $E$ and the intersection number is invariant under deformation,
we have that
$$(K_{F'}+E_{t_0}) \cdot \omega <0.$$

{\it Claim 2} $C_0 $ is a $(-1)$-curve in $F'$ and the $C_i, i \geq 1,$ are $(-2)$-curves.

Define
$$\cI:= \otimes \cO(-C_i)^{\otimes a_i} \subset \cO.$$
Consider the map
$$H^1(F', \cO_{F'}) \to H^1(F', \cO_{F'} / \cI) \to H^2(F', \cI)=H^0(F',K_{F'}+\sum a_i C_i )=0$$
by Serre duality.
The last equality follows from the inequality $(K_{F'}+C) \cdot \omega <0.$
Since $F'$ is rational, $H^1(F', \cO_{F'} / \cI)=H^1(F', \cO_{F'} )=0.$

Since $H^1(E_{t_0}, \cO_{E_{t_0}})=0$, any irreducible component is a smooth rational curve.
Adjunction formula shows that
$C_0 $ is a $(-1)-$curve in $F'$ and the $C_i, i \geq 1,$ are $(-2)-$curves.

{\it Claim 3} $E_{t_0}$ has only one irreducible component.

Now we look at the deformations of $C_0$ and obtain a family
$(C_s)_{s \in A.}$ For a small neighborhood $\Delta' \subset Y'$ of $t_0$ the curve $C_t $ is in $\pi^{-1}(\varphi(t))$
(strictly speaking there is a canonical map $f: A \to Y,$ and $f\vert_\Delta $ is biholomorphic on a small neighborhood $\Delta \subset A$ of $0$, so that
we can identify $t$ and $f(t)$ for small $t).$
Therefore we can consider the (non-effective) family of cycles $(E_t - C_t)_{t \in \Delta'}$ so that
$$ E_{t_0} - C_0 = \sum_{i \geq 1} a_i C_i.$$
By the choice of $E_t,$ $\omega \cdot E_t$ is minimal for general $t,$ therefore $\omega \cdot E_t \leq \omega \cdot C_t$ and
we conclude
$$ \omega \cdot \sum_{i \geq 1} a_i C_i = 0$$
and therefore $a_i = 0$ for $i \geq 1$ so that $E_{t_0}$ is irreducible and reduced.

{\it Claim 4} $Y'$ is smooth at $t_0.$

Note that $E_{t_0}$ has no embedded point.
In fact, an embedded point would increase
the arithmetic genus, which would then contradict the fact that the arithmetic genus
stays 0 in the flat family and that the reduced underlying space is a rational curve.
Note that the reduced underlying space of $E_{t_0}$ is a smooth $(-1)$-curve.
Thus we have
$$ h^1(E_{t_0},N_{E_{t_0}/ X}) = 0$$
and $Y'$ is smooth at $t_0.$

In conclusion, $Y'$ is smooth and $M \to Y'$ is a $\P^1$-bundle.
The
fact that any fiber of $Y'$ represents some $(-1)$-curve in some fiber of $\pi$ implies that $Y' \to Y$ is \'etale.
So set $D = Y'$ and
define $S$ to be the irreducible component of $  \mathrm{pr}_1(M) \times_Y Y'$ mapping onto $\mathrm{pr}_1(M).$
\end{proof}
Note that the assumption is necessary for some well-chosen K\"ahler form on $X$.
The assumption is satisfied if $(K_F)^2 \geq 2$ and $-K_X$ is nef.
In fact, assume on the contrary that for any K\"ahler form $\omega$ and any $(-1)-$curve $C$ in any fiber $F$ we have
$$(K_F+C) \cdot \omega \geq 0.$$
In particular, the intersection number of $(K_F+C) $ with any nef class is non-negative.
Thus
$$(K_F+C) \cdot(-K_X)|_F=(K_F+C) \cdot (-K_F)=-(K_F)^2 +1\geq 0$$
which is impossible.

\begin{rem}
\label{Blow-Down}
 Note that the proof of Proposition \ref{PS04} holds if we replace the K\"ahler form $\omega$ by some big class $\alpha$ such that the restriction of $\alpha$ to any fiber $F$ is big and nef in codimension 1 (cf. e.g. \cite{nef}).
 In particular, by Demailly's theory of Monge-Amp\`ere operators  (see e.g. \cite[Section 4, Chap III]{agbook}), for any effective curve $C$ in $F$,
 the intersection number $C \cdot \alpha|_F$ is strictly positive.

 In particular, if $(K_F)^2 \geq 2$, $-K_X$ is nef and the other assumptions of Proposition \ref{PS04} are satisifed, up to some finite \'etale cover, we can blow down $\pi: X \to X'$ to some compact manifold $X'$ of class $\mathcal{C}$ which admits a smooth Del Pezzo fibration of degree at least 3 onto another compact manifold of class $\mathcal{C}$.
 Note that $\pi_*(-K_X+\epsilon \omega)$ is a big class for any $\epsilon >0$ whose restriction to any Del Pezzo surface fiber is big and nef in codimension 1.
 Thus, if the Del Pezzo surface fiber is not $\P^2$ or a product $\P^1 \times \P^1$, one can apply again Proposition \ref{PS04} to get a further blown-down (after passing to a suitable finite \'etale cover).
\end{rem}
The following proposition is a special case of \cite{CH17}, cf. also \cite[Corollary 3.4]{PS98}.
However, we find that its proof has its own interest using only deformation theory under the additional assumption that the Albanese morphism is smooth and the condition on the intersection number. 
\begin{prop}
\label{PS3.4}
 Let $X$ be a uniruled compact K\"ahler manifold of dimension $n$ such that
$\tilde{q}(X)=q(X)=n-2$ and
$-K_X$ is nef.
Assume that the Albanese morphism $\alpha: X \to A(X)$ is smooth.
Moreover, assume that for any fiber $F$ of $\alpha$,
$$(K_F)^2 \geq 2.$$
Then, up to some finite \'etale cover, the possibilities of $X$ are one of the following:
\begin{enumerate}
 \item $X \simeq \P(E)$ for some numerically flat vector bundle (of rank 3) over $A(X)$;
 \item $X \simeq \P(E_1) \times_{A(X)} \P(E_2)$ for some numerically flat vector bundles $E_1$ and $E_2$ (of rank 2) over $A(X)$;
 \item we have a factorisation
$$X \xrightarrow{\gamma_1} X_1 \cdots \xrightarrow{\gamma_k} X_k \xrightarrow{\rho}  A(X)$$
with $\gamma_j$ blow-ups of  \'etale multi-sections, $X_k \simeq \P(E_1) \times_{A(X)} \P(E_2)$ for some vector bundles $E_1,E_2$ over $A(X)$ or $X_k \simeq \P(E)$ for some vector bundle over $A(X)$.
\end{enumerate}
\end{prop}
\begin{proof}
First note that all non-minimal surfaces $F$ with $-K_F$ nef are isomorphic to the
plane $\P^2$ blown up in at most 9 points in sufficiently general position (by the Kodaira-Enriques classification).
All minimal rational ones are either $\P^2$ or $\P^1 \times \P^1$.
The case of $\P^2$ is treated as follows:
By \cite[Lemma 7.4]{CP91}, up to some finite \'etale cover, we have $X \simeq \P(E) $ for some vector bundles $E$ over $A(X)$ with trivial determinant.
Since $-K_X$ is nef, $E$ is nef, hence numerically flat. 

The case of $\P^1 \times \P^1$ can be treated similarly:
Notice that the proof of \cite[Lemma 9.4]{CP91} works also in the compact K\"ahler setting.
Up to some finite \'etale cover, we have $X \simeq \P(E_1) \times_{A(X)} \P(E_2)$ for some vector bundles $E_1,E_2$ over $A(X)$ with trivial determinant.
By positivity of direct images applied to the submersions $X \to \P(E_i)$ $(i=1,2)$ with the condition that $-K_X$ is nef,
$\P(E_i)$ $(i=1,2)$ has nef anticanonical line bundle which implies that $E_1$ and $E_2$ are numerically flat.

The remaining case can be treated as follows:
By the preceding remark \ref{Blow-Down}, the factorisation as stated in case (3) exists.
Thus, $X_k$ is a compact manifold of class $\mathcal{C}$
whose fibers of the Albanese morphism are either $\P^2$ or $\P^1 \times \P^1$.
The conclusion follows from \cite[Lemma 9.4]{CP91} and \cite{El82}.
Note that proof of \cite[Lemma 9.4]{CP91} also works for compact manifolds of class $\mathcal{C}$ by replacing the Chow variety by the corresponding Barlet space whose irreducible components are known to be compact.
In particular, all $X_i$'s are in fact K\"ahler.
Note that the Albanese morphism is always projective in this case.
(In fact in this case, the Albanese morphism can be chosen to be a MRC quotient map.)

We claim that the Albanese morphism is locally constant (cf. \cite[Definition 2.3]{MW}) by applying \cite[Proposition 2.5]{MW}.
Since all spaces are smooth, it is enough to show the existence of a relatively ample line bundle whose direct image is numerically flat.
By \cite[Theorem 2.6]{MW23}, it is enough to show the existence of a relatively ample and pseudoeffective line bundle.
Denote by $\pi: X \to X_k$ the composition with exceptional divisor $E$.
Since $-K_{X_k}$ is relatively ample with respect to the natural projection $X_k \to A(X)$, for $m \gg 0$, $m \pi^* (-K_{X_k})-E$ is $\alpha-$ample.
On the other hand, $-K_X=\pi^*(-K_{X_k})-E$ which implies that $m \pi^* (-K_{X_k})-E=-mK_X+(m-1)E$ is pseudoeffective.
\end{proof}

\begin{rem}
We find that there is a gap in the proof of \cite[Proposition 0.4]{PS98}.
The following arguments are given in the proof:
Since $F'$ is realised as blow-up of $\P^2$ in 9 points, we may take a general line $l$ in $F'.$
This general line $l$ can be deformed in $X$ to a general line in a neighbouring $F_s$.
For general s we have
$$ (K_{F_s} + E_s) \cdot l_s < 0,$$
where $E_s$ is one of the $(-1)-$curves in our family sitting inside $F_s.$
The following example shows that the intersection number of $(K_{F_s} + E_s)$ with $l_s$ can be 0 for some $(-1)-$curve.
This comes from the fact that $F'$ may have different contractions to $\P^2$.
Note that the result itself holds as a consequence of structure theorem of projective manifolds with nef anticanonical line bundle as shown in \cite{CH19}.
\end{rem}

\begin{ex}
Let $S_8$ be a Del Pezzo surface of degree 1.
Let $\pi: S_8 \to \P^2$ be the blow down with exceptional curves $E_i(1 \leq i \leq 8)$.
By \cite[Proposition 8.2.19]{Dol09}
there exist $(-1)-$curves $C_i(1 \leq i \leq 8)$ on $S_8$ such that
$C_i+E_i=-2K_{S_8}$.
Then
$$\pi_* C_i=-2K_{\P^2}=\cO(6)$$
With the choice $F_s=S_8$, $E_s=C_i$, $l_s$ the strict transforms of $\cO(1)$,
$$ (K_{F_s} + E_s) \cdot l_s \geq 0.$$
(In fact it is enough to choose a Del Pezzo surface of degree 2.)

\end{ex}

\begin{rem}
The above proof of Proposition \ref{PS3.4} holds for a smooth fibration $\pi:X \to Y$ from a compact K\"ahler fourfold $X$ to a K3 surface $Y$ under the assumption on the intersection number.
In this case, there are three possiblities:
\begin{enumerate}
\item $Y \times \P^2$;
\item $Y \times \P^1 \times \P^1$;
\item we have a factorisation
$$X \xrightarrow{\gamma_1} X_1 \cdots \xrightarrow{\gamma_k} X_k \xrightarrow{\rho}  Y$$
with $\gamma_j$ blow-ups of  \'etale multi-sections, $X_k \simeq \P(E_1) \times_{Y} \P(E_2)$ for some vector bundles $E_1,E_2$ over $Y$ or $X_k \simeq \P(E)$ for some vector bundle over $Y$.
\end{enumerate}
More precisely, for the first  and second case, the anticanonical line bundle of the fiber $-K_F$ is very ample.
By the Kodaira vanishing theorem,
$H^1(F, -mK_F)=0$, so $-K_X$ is $\pi-$ample by \cite[Theorem 4.2]{BS76}. 
Then we can argue as in \cite{CH19} (cf. e.g. \cite[Theorem 2.6]{MW23}), to see that $\pi_*(-K_X)$ is numerically flat.
By \cite[Proposition 2.5]{MW}, the conclusion follows.
(Note that the Brauer group theory implies that there exists a $\P^r-$bundle over a K3 surface which cannot be written as the projectivization of some vector bundle.)
The remaining case follows similarly as the end of Proposition \ref{PS3.4}.

Conjecturally, all compact K\"ahler non-projective fourfolds $X$ with nef anticanonical line bundle are, up to a finite \'etale cover, one of the following types:
\begin{enumerate}
\item $c_1(X)=0$;
\item there is a locally trivial fibration $\pi:X \to Y$ from a compact K\"ahler fourfold $X$ to a K3 surface or a two dimensional torus $Y$;
\item $X=\P(E)$ for some rank 2 numerically flat vector bundle over a torus of dimension three;
\item the product of a
K3 surface and a projectivization of a rank two numerically flat vector bundle over an
elliptic curve.
\end{enumerate}
The present unknown cases are the simply connected ones with trivial canonical line bundle and case (2) without assuming that the fibration is smooth and having the assumption on the intersection number.
\end{rem}

\section{Algebraic approximation}

The following lemma is well-known to experts. 
For convenience of the reader, we give a proof here which was communicated to us by Campana.
\begin{lem}\label{Lem8}
Let $X$ be a compact K\"ahler manifold.
Assume that the rational quotient $R(X)$ of $X$ is projective.
Then $X$ is projective.
\end{lem}
\begin{proof}
By \cite{Cam81}, to show that $X$ is projective, it is enough to show that any two general points of $X$ can be connected by a chain of (compact) curves.
Without loss of generality, we can assume that the rational quotient is a holomorphic map since as a compact K\"ahler manifold, the projectivity of any modification of $X$ would imply that $X$ is Moishezon (and hence projective).

Take a very ample line bundle $A$ over $R(X)$ which exists by the projectivity assumption.
By Bertini's theorem, any two general points of $R(X)$ can be connected by a smooth curve $C$ defined as a complete intersection of sections in the linear system of $A$ such that the general fiber over $C$ is smooth and rationally connected.
By Sard's theorem, we can furthermore assume that the base change $X_C$ is smooth.
It is enough to show that $X_C$ is projective which implies that any two points of $X_C$ can be connected by a chain of (compact) curves.

By Kodaira's criterion, as $X_C$ is a compact K\"ahler manifold, it is enough to show that $H^0(X_C, \Omega_{X_C}^2)=0$.
Let $C_0 \subset C$ be the smooth locus such that each fiber is a rationally connected manifold.
Any section vanishing over a non-empty Zariski open set of $X_C$ has to vanish identically.
Therefore, it is enough to show that $H^0(X_{C_0}, \Omega_{X_{C_0}}^2)=0$.
We denote by $f: X_{C_0} \to C_0$ the restricted map which gives a short exact sequence of vector bundles
$$0 \to f^* \Omega^1_{C_0} \to \Omega^1_{X_{C_0}} \to \Omega^1_{X_{C_0}/C_0} \to 0$$
that induces a filtration on $\Omega^2_{X_{C_0}}:$
$$0 \subset f^*\Omega^2_{{C_0}} \subset f^*\Omega^1_{{C_0}} \otimes \Omega^1_{X_{C_0}} \subset \Omega^2_{X_{C_0}}.$$
The graded pieces are
$$R_0:=f^*\Omega^2_{{C_0}},R_1:= \Omega^1_{X_{C_0}/C_0} \otimes f^* \Omega^1_{{C_0}}, R_2:=\Omega^2_{X_{C_0}/C_0}.$$
It is enough to show that the graded pieces have no global sections.
$R_0=0$ since $C$ is of dimension 1.
$H^0(X_{C_0}, \Omega^2_{X_{C_0}/C_0})=H^0(C_0, f_* \Omega^2_{X_{C_0}/C_0})=0$ since by Grauert's theorem, for any $z \in C_0$,
$H^0(X_z, \Omega^2_{X_{C_0}/C_0}|_{X_z})=0$.
Note that a rationally connected manifold has no non-trivial holomorphic forms.
$H^0(X_{C_0}, \Omega^1_{X_{C_0}/C_0} \otimes f^* \Omega^1_{{C_0}})=H^0(C_0, f_*(\Omega^1_{X_{C_0}/C_0} \otimes f^* \Omega^1_{{C_0}}))=H^0(C_0, f_*(\Omega^1_{X_{C_0}/C_0}) \otimes \Omega^1_{{C_0}}))=0$
 where the last equality follows from Grauert's theorem.
This finishes the proof of the lemma.
\end{proof}
\begin{rem}
Similar arguments in fact show the following statement:
Let $f:X\to B$ be a proper holomorphic submersion with connected fibers between two  connected complex manifolds. 
If the fibers are rationally connected we have
$$H^0(X,\Omega^p_X)=f^*(H^0(B,\Omega^p_B))$$ for any $p\geq 0$.
\end{rem}

\begin{prop}
\label{AlgApprox}
Assume that Conjecture 2 holds. Then any compact K\"ahler manifold $X$ with nef anticanonical divisor can be algebraically approximated.
\end{prop}
\begin{proof}
Consider the fibration $\varphi: X \to Y$ predicted in Conjecture 1.
By the well-known fact that a Calabi-Yau manifold can be algebraically approximated, there exists a deformation
$\pi: \mathfrak{Y} \to \Delta$ of $Y$ over a small disk $\Delta$ such that for Euclidean dense $t \in \Delta$, $Y_t := \pi^{-1}(t)$ is a projective Calabi-Yau manifold.
Let $p: \tilde{\mathfrak{Y}} \to \mathfrak{Y}$ be the (holomorphic) universal cover of $\mathfrak{Y}$.

We claim that $\tilde{Y}_t := (p \circ \pi)^{-1}(t)$ is the universal cover of $Y_t$ for any $t$.
Since a base change of \'etale morphism is still \'etale, the projection $\tilde{Y}_t \to Y_t$ is \'etale (i.e. a unbranched covering map).
Since $\Delta$ is contractible, by Ehrensmann's theorem,
there exist $2 $ pointwise linear independent smooth vector fields over $\mathfrak{Y}$.
These vector fields define a diffeomorphism between $\mathfrak{Y}$ and $Y_0 \times \Delta$.
Since $p$ is \'etale, the pullback of these vector fields are well defined on $\tilde{\mathfrak{Y}}$.
Up to a possible shrinking of $\Delta$, the pullback fields define a diffeomorphism between $\tilde{\mathfrak{Y}}$ and $\tilde{Y}_0 \times \Delta$.
In particular, for any $t$, $\pi_1(\tilde{Y}_t)=\pi_1(\tilde{\mathfrak{Y}})=0$ since $\Delta$ is contractible.
This finishes the proof of the claim.

Consider $ \mathfrak{X} := (\tilde{\mathfrak{Y}} \times F)/\pi_1(Y) \xrightarrow{\tilde{\varphi}} \mathfrak{Y} \to \Delta$ which defines a deformation of $X$.
Since the group action is free, $\mathfrak{X}$ is smooth.
For any $t \in \Delta$, the morphism $X_t:= (\pi \circ \tilde{\varphi})^{-1}(t) \to Y_t$ has rational connected fiber whose base is non-uniruled.
Thus $Y_t$ is a MRC quotient of $X_t$.
For dense $t$, Lemma \ref{Lem8} implies that $X_t$ is projective which finishes the proof.
Note that for such dense $t$, $-K_{X_t}$ is nef.
Indeed, since $-K_X$ is nef, by definition, for any $\epsilon >0$, there exists a smooth metric $h_\epsilon$ such that its Chern curvature satisfies $c_1(-K_X, h_\epsilon) \geq -\epsilon \omega$.
We pull back $h_\epsilon$ to the universal cover $q: \tilde{Y} \times F \to X$ of $X$ whose restriction to $\{ \tilde{y} \} \times F$ for any $\tilde{y} \in \tilde{Y}$ defines a $\pi_1(Y)-$invariant metric on $-K_F$.
The product metric defined by the restricted $h_\epsilon$ and the Ricci flat metric on $\tilde{Y}_t$ is $\pi_1(Y)-$invariant and thus descends to $X_t$ which implies that $-K_{X_t}$ is nef.
\end{proof}
\begin{rem}
The above Proposition shows that the nefness of the anticanonical line bundle is a deformation invariant property.
However, the strong pseudoeffectivity of the tangent bundle is not a deformation invariant property in general.
For example, the blow-up of $\P^2$ in at most 4 points (in general position) has a strongly pseudoeffective tangent bundle as shown in \cite{HIM}.
However, any toric variety has strongly pseudoeffective tangent bundle which can be taken as blow-up of $\P^2$ at arbitrary many points (in special position).
\end{rem}

\end{document}